\documentclass{amsart}      
\usepackage{amssymb,amsthm, amsmath, amsfonts} 
\usepackage{graphics}                 
\usepackage{hyperref}                 
\usepackage[all]{xy}
\usepackage{enumerate}

\newtheorem{theorem}{Theorem}[section]
\newtheorem{lemma}[theorem]{Lemma}
\newtheorem{proposition}[theorem]{Proposition}

\newtheorem{corollary}[theorem]{Corollary}

\newtheorem{remark}[theorem]{Remark}
\newtheorem{conjecture}[theorem]{Conjecture}
\numberwithin{equation}{section}

\DeclareMathOperator{\pd}{pd}

\DeclareMathOperator{\gldim}{gl.dim}
\DeclareMathOperator{\fdim}{fin.dim}
\DeclareMathOperator{\sgldim}{sgl.dim}

\title[Skew group algebras]{Finitistic dimensions and piecewise hereditary property of skew group algebras}
\author{Liping Li}
\address{Department of Mathematics, University of California, Riverside, CA, 92521.}
\email{lipingli@math.ucr.edu}
\thanks{The author would like to thank the referee for carefully reading and checking this paper, and for his/her valuable comments.}

\begin{document}

\begin{abstract}
Let $\Lambda$ be a finite dimensional algebra and $G$ be a finite group whose elements act on $\Lambda$ as algebra automorphisms. Under the assumption that $\Lambda$ has a complete set $E$ of primitive orthogonal idempotents, closed under the action of a Sylow $p$-subgroup $S \leqslant G$. If the action of $S$ on $E$ is free, we show that the skew group algebra $\Lambda G$ and $\Lambda$ have the same finitistic dimension, and have the same strong global dimension if the fixed algebra $\Lambda^S$ is a direct summand of the $\Lambda^S$-bimodule $\Lambda$. Using a homological characterization of piecewise hereditary algebras proved by Happel and Zacharia, we deduce a criterion for $\Lambda G$ to be piecewise hereditary.
\end{abstract}

\maketitle

\section{Introduction}

Throughout this note let $\Lambda$ be a finite dimensional $k$-algebra, where $k$ is an algebraically closed field with characteristic $p \geqslant 0$, and let $G$ be a finite group whose elements act on $\Lambda$ as algebra automorphisms. The \textit{skew group algebra} $\Lambda G$ is the vector space $\Lambda \otimes_k kG$ equipped with a bilinear product $\cdot$ determined by the following rule: for $\lambda, \mu \in \Lambda$, $g, h \in G$, $(\lambda \otimes g) \cdot (\mu \otimes h) = \lambda g(\mu) \otimes gh$, where $g(\mu)$ is the image of $\mu$ under the action of $g$. We write $\lambda g$ rather than $\lambda \otimes g$ to simplify the notation. Correspondingly, the product can be written as $\lambda g \cdot \mu h = \lambda g(\mu) gh$. Denote the identity of $\Lambda$ and the identity of $G$ by $1_{\Lambda}$ and $1_G$ respectively.

It has been observed that when $|G|$, the order of $G$, is invertible in $k$, $\Lambda G$ and $\Lambda$ share many common properties \cite{Dionne, Martinez, Reiten}. We wonder for arbitrary groups $G$, under what conditions this phenomen still happens. This problem is considered in \cite{Li}, where under the hypothesis that $\Lambda$ has a complete set $E$ of primitive orthogonal idempotents closed under the action of a Sylow $p$-subgroup $S \leqslant G$, we show that $\Lambda$ and $\Lambda G$ share certain properties such as finite global dimension, finite representation type, etc., if and only if the action of $S$ on $E$ is free. Clearly, this answer generalizes results in \cite{Reiten} since if $|G|$ is invertible in $k$, the only Sylow $p$-subgroup of $G$ is the trivial group.

In this note we continue to study representations and homological properties of modular skew group algebras. Using the ideas and techniques described in \cite{Li}, we show that $\Lambda$ and $\Lambda G$ share more common properties under the same hypothesis and condition. Explicitly, we have:

\begin{theorem}
Let $\Lambda$ and $G$ be as above, and let $S \leqslant G$ be a Sylow $p$-subgroup. $\Lambda$ has a complete set $E$ of primitive orthogonal idempotents closed under the action of $S$. Then:
\begin{enumerate}
\item If the action of $S$ on $E$ is free, then $\Lambda G$, $\Lambda$, and $\Lambda^S$ (the fixed algebra by $S$) have the same finitistic dimension.
\item $\Lambda G$ has finite strong global dimension if and only if $\Lambda ^S$ has finite strong global dimension and $S$ acts freely on $E$. In this situation, $\Lambda G$ and $\Lambda ^S$ have the same strong global dimension; moreover, if $\Lambda \cong \Lambda^S \oplus B$ as $\Lambda^S$-bimodules, then $\Lambda S$ and $\Lambda$ have the same strong global dimension.
\end{enumerate}
\end{theorem}

In \cite{Dionne} it has been proved that if $\Lambda$ is a piecewise hereditary algebra (defined in Section 3) and $|G|$ is invertible in $k$, then $\Lambda G$ is piecewise hereditary as well. The second part of the above theorem generalizes this result by using the homological characterization of piecewise hereditary algebras by Happel and Zacharia in \cite{Happel2}

We introduce some notations and conventions here. Throughout this note all modules are finitely generated left modules. Composition of maps and morphisms is from right to left.  For an algebra $A$, $\gldim A$, $\fdim A$, and $\sgldim A$ are the global dimension, finitistic dimension, and strong global dimension (defined in Section 4) of $A$ respectively. For an $A$-module $M$, $\pd_A M$ is the projective dimension of $M$. We use $A$-mod to denote the category of finitely generated $A$-modules. Its bounded homotopy category and bounded derived category are denoted by $K^b(A)$ and $D^b(A)$ respectively.

\section{projective dimensions and finitistic dimensions}

We first describe some background knowledge and elementary results. Most of them can be found in literature. We suggest the reader to refer to \cite{Auslander, Boisen, Cohen, Li, Marcus, Martinez, Passman, Reiten, Zhong} for more details.

For every subgroup $H \leqslant G$, elements in $H$ also act on $\Lambda$ as algebra automorphism, so we can define a skew group algebra $\Lambda H$, which is a subalgebra of $\Lambda G$. The induction functor and the restriction functor are defined in the usual way. For a $\Lambda H$ module $V$, the induced module is $V \uparrow _H^G = \Lambda G \otimes _{\Lambda H} V$, where $\Lambda G$ acts on the left side. Every $\Lambda G$-module $M$ can be viewed as a $\Lambda H$-module, denoted by $M \downarrow _H^G$. Observe that $\Lambda G$ is a free (both left and right) $\Lambda H$-module. Therefore, these two functors are exact, and perverse projective modules.

\begin{proposition}
Let $H \leqslant G$ be a subgroup. Then:
\begin{enumerate}
\item Every $\Lambda H$-module $V$ is isomorphic to a summand of $V \uparrow _H^G \downarrow _H^G$.
\item If $|G:H|$ is invertible in $k$, then every $\Lambda G$-module $M$ is isomorphic to a summand of $M \downarrow _H^G \uparrow _H^G$.
\end{enumerate}
\end{proposition}

\begin{proof}
This is Proposition 2.1 in \cite{Li}.
\end{proof}

The above proposition immediately implies:

\begin{corollary}
Let $H \leqslant G$ be a subgroup. For every $M \in \Lambda G$-mod, $\pd _{\Lambda H} M \downarrow _H^G \leqslant \pd_{\Lambda G} M$. If $|G: H|$ is invertible in $k$, then the equality holds.
\end{corollary}

\begin{proof}
Take a minimal projective resolution $P^{\bullet} \rightarrow M$ of $\Lambda G$-modules and apply the restriction functor termwise. Since this functor is exact and preserves projective modules, we get a projective resolution $P^{\bullet} \downarrow _H^G \rightarrow M \downarrow _H^G$ of $\Lambda H$-modules, which might not be minimal. Thus $\pd _{\Lambda H} M \downarrow _H^G \leqslant \pd_{\Lambda G} M$, which is true even for the case that either $\pd _{\Lambda H} M \downarrow _H^G = \infty$ or $\pd_{\Lambda G} M = \infty$.

Now suppose that $|G:H|$ is invertible in $k$. Take a minimal projective resolution $Q^{\bullet} \rightarrow M \downarrow _H^G$ of $\Lambda H$-modules and apply the induction functor termwise. Since it is exact and preserves projective modules, we get a projective resolution $Q^{\bullet} \uparrow _H^G \rightarrow M \downarrow _H^G \uparrow _H^G$ of $\Lambda G$-modules, so $\pd _{\Lambda G} M \downarrow _H^G \uparrow _H^G \leqslant \pd_{\Lambda H} M \downarrow _H^G$. But $M$ is isomorphic to a summand of $M \downarrow _H^G \uparrow _H^G$, so $\pd _{\Lambda G} M \leqslant \pd_{\Lambda H} M \downarrow _H^G$. Putting two inequalities together, we get the equality.
\end{proof}

Let $E = \{ e_i \} _{i \in [n]}$ be a set of primitive orthogonal idempotents in $\Lambda$. We say it is \textit{complete} if $\sum _{i \in [n]} e_i = 1_{\Lambda}$. Throughout this note we assume that $G$ has a Sylow $p$-subgroup $S$ such that $E$ is closed under the action of $S$. That is, $g(e_i) \in E$ for all $i \in [n]$ and $g \in S$. In practice this condition is usually satisfied. A trivial case is that $|G|$ is invertible in $k$, and hence $S$ is the trivial subgroup.

We introduce some notations here. Let $\Lambda ^S$ be the space consisting of all elements in $\Lambda$ fixed by $S$. Clearly, $\Lambda ^S$ is a subalgebra of $\Lambda$, and $S$ acts on it trivially. For every $M \in \Lambda S$-mod, elements $v \in M$ satisfying $g (v) = v$ for every $g \in S$ form a $\Lambda ^S$-module, which is denoted by $M^S$. Let $F^S$ be the functor from $\Lambda S$-mod to $\Lambda ^S$-mod, sending $M$ to $M^S$. This is indeed a functor since for every $\Lambda S$-module homomorphism $f: M \rightarrow N$, $M^S$ is mapped into $N^S$ by $f$.

We collect in the following proposition some results taken from \cite{Li}.

\begin{proposition}
Let $S \leqslant G$ and $E$ be as above and suppose that $E$ is closed under the action of $S$. Then:
\begin{enumerate}
\item The set $E$ is also a complete set of primitive orthogonal idempotents of $\Lambda S$, where we identify $e_i$ with $e_i 1_S$.
\item $\Lambda ^S = \{ \sum_{g \in S} g(\mu) \mid \mu \in \Lambda \}$.
\item The global dimension $\gldim \Lambda S < \infty$ if and only if $\gldim \Lambda < \infty$ and the action of $S$ on $E$ is free.
\end{enumerate}
Moreover, if the action of $S$ on $E$ is free, then
\begin{enumerate}
\setcounter{enumi}{3}
\item $\Lambda S$ is a matrix algebra over $\Lambda^S$, and hence is Morita equivalent to $\Lambda^S$.
\item The functor $F^S$ is exact.
\item The regular representation $_{\Lambda S} \Lambda S \cong \Lambda ^{|S|}$, where $\Lambda$ is the trivial $\Lambda S$-module.
\item A $\Lambda S$-module $M$ is projective (resp., injective) if and only if the $\Lambda$-module $_{\Lambda} M$ is projective (resp., injective).
\end{enumerate}
\end{proposition}

\begin{proof}
See Lemmas 3.1, 3.2, 3.6 and Propositions 3.3, 3.5 in \cite{Li}.
\end{proof}

Recall for a finite dimensional algebra $A$, its \textit{finitistic dimension}, denoted by $\fdim A$, is the supremum of projective dimensions of finitely generated indecomposable $A$-modules $M$ with $\pd_A M < \infty$. If $A$ has finite global dimension, then $\fdim A = \gldim A$. The famous finitistic dimension conjecture asserts that the finitistic dimension of every finite dimensional algebra is finite. For more details, see \cite{Zimmermann}.

An immediate consequence of Proposition 2.1 and Corollary 2.2 is:

\begin{lemma}
Let $H \leqslant G$ be a subgroup of $G$. Then $\fdim \Lambda H \leqslant \fdim \Lambda G$. If $|G:H|$ is invertible in $k$, the equality holds.
\end{lemma}

\begin{proof}
Take an arbitrary indecomposable $V \in \Lambda H$-mod with $\pd _{\Lambda H} V < \infty$ and consider $\tilde{V} = V \uparrow _H^G$. We claim that $\pd _{\Lambda G} \tilde{V} < \infty$. Indeed, by applying the induction functor to a minimal projective resolution $P^{\bullet} \rightarrow V$ termwise, we get a projective resolution $\Lambda G \otimes _{\Lambda H} P^{\bullet} \rightarrow \tilde{V}$. The finite length of the first resolution implies the finite length of the second one. Therefore, $\pd _{\Lambda G} \tilde{V} < \infty$. Consequently, $\pd _{\Lambda G} \tilde{V} \leqslant \fdim \Lambda G$.

By Corollary 2.2, $\pd_{\Lambda G} \tilde{V} \geqslant \pd_{\Lambda H} \tilde{V} \downarrow _H^G$. Since $V$ is isomorphic to a summand of $\tilde{V} \downarrow _H^G$ by Proposition 2.1, we get $\fdim A \geqslant \pd _{\Lambda G} \tilde{V} \geqslant \pd_{\Lambda H} \tilde{V} \downarrow _H^G \geqslant \pd _{\Lambda H} V$. In conclusion, $\fdim \Lambda G \geqslant \fdim \Lambda H$.

Now suppose that $|G : H|$ is invertible in $k$. Take an arbitrary indecomposable $M \in \Lambda G$-mod with $\pd _{\Lambda G} M < \infty$. By applying the restriction functor $\downarrow _H^G$ to a minimal projective resolution $Q^{\bullet} \rightarrow M$ termwise, we get a projective resolution of $M \downarrow _H^G$, which is of finite length. Therefore, $\pd _{\Lambda H} M \downarrow _H^G \leqslant \fdim \Lambda H$. By Corollary 2.2, we have $\pd _{\Lambda G} M = \pd _{\Lambda H} M \downarrow _H^G \leqslant \fdim \Lambda H$. In conclusion, $\fdim \Lambda G \leqslant \fdim \Lambda H$. Combining this with the inequality in the previous paragraph, we have $\fdim \Lambda G = \fdim \Lambda H$.
\end{proof}

\begin{lemma}
If a Sylow $p$-subgroup $S \leqslant G$ acts freely on $E$, then $\fdim \Lambda S \leqslant \fdim \Lambda$.
\end{lemma}

\begin{proof}
Take an arbitrary indecomposable $M \in \Lambda S$-mod with $\pd _{\Lambda S} M = n < \infty$ and a minimal projective resolution:
\begin{equation*}
\ldots \rightarrow 0 \rightarrow P^n \rightarrow \ldots \rightarrow P^1 \rightarrow P^0 \rightarrow M \rightarrow 0.
\end{equation*}
Regarded as $\Lambda$-modules, we get a projective resolution:
\begin{equation*}
\ldots \rightarrow 0 \rightarrow _{\Lambda} P^n \rightarrow \ldots \rightarrow _{\Lambda} P^1 \rightarrow _{\Lambda} P^0 \rightarrow _{\Lambda}M \rightarrow 0.
\end{equation*}
Clearly, for every $0 \leqslant s \leqslant n$, the syzygy $\Omega^s (M)$ viewed as a $\Lambda$-module is a direct sum of $\Omega^s (_{\Lambda} M)$ and a projective $\Lambda$-module. We claim that for each $0 \leqslant s \leqslant n$, the syzygy $\Omega^s (_{\Lambda} M) \neq 0$. Otherwise, $\Omega^s (M)$ viewed as a $\Lambda$-module is projective. By (7) of Proposition 2.3, $\Omega^s (M)$ is a projective $\Lambda S$-module, which must be 0. But this implies $\pd _{\Lambda S} M < n$, contradicting our choice of $M$. Therefore, for each $0 \leqslant s \leqslant n$, $\Omega^s (_{\Lambda} M) \neq 0$. Consequently, $_{\Lambda} M$ has a summand with projective dimension $n$, so $n \leqslant \fdim \Lambda$. In conclusion, $\fdim \Lambda S \leqslant \fdim \Lambda$.
\end{proof}

Now we can prove the first statement of Theorem 1.1.

\begin{proposition}
If $G$ has a Sylow $p$-subgroup $S$ acting freely on $E$, then $\Lambda G$ and $\Lambda$ have the same finitistic dimension.
\end{proposition}

\begin{proof}
Since $|G : S|$ is invertible in $k$, by Lemma 3.1, $\fdim \Lambda G = \fdim \Lambda S$. Also by this lemma, $\fdim \Lambda S \geqslant \fdim \Lambda$ since $S$ contains the trivial group. The previous lemma tells us $\fdim \Lambda S \leqslant \fdim \Lambda$. Putting all these pieces of information together, we get $\fdim \Lambda G = \fdim \Lambda$ as claimed.
\end{proof}

From the previous proposition we conclude immediately that if the action of $S$ on $E$ is free, then the finiteness of finitistic dimension of $\Lambda$ implies the finiteness of finitistic dimension of $\Lambda G$. We wonder whether this conclusion is true in general. That is,

\begin{conjecture}
Let $\Lambda, G, E$ and $S$ be as before and suppose that $E$ is closed under the action of $S$. If $\fdim \Lambda < \infty$ (or even stronger $\gldim \Lambda < \infty)$, then $\fdim \Lambda G < \infty$.
\end{conjecture}

Hopefully the proof of this question can shed light on the final proof of the finitistic dimension conjecture.

\section{strong global dimensions and piecewise hereditary algebras}

A finite dimensional $k$-algebra $A$ is called \textit{piecewise hereditary} if the derived category $D^b(A)$ is equivalent to the derived category $D^b(\mathcal{H})$ of a hereditary abelian category $\mathcal{H}$ as triangulated categories; see \cite{Happel1, Happel2, Happel3, HRS}. If $A$ is piecewise hereditary, then $\gldim A < \infty$ since the property having finite global dimension is invariant under derived equivalence \cite{Happel1}. Therefore, $D^b(A)$ is triangulated equivalent to the homotopy category $K^b (_A \mathcal{P})$, where $_A \mathcal{P}$ is the full subcategory of $A$-mod consisting of all finitely generated projective $A$-modules.

Now let $A$ be a finite dimensional algebra with $\gldim A < \infty$. Since $D^b(A) \cong K^b (_A\mathcal{P})$, we identify these two triangulated categories. For every indecomposable object $0 \neq P^{\bullet} \in K^b(_A \mathcal{P})$, consider its preimages $\tilde{P}^{\bullet} \in C^b (_A\mathcal{P})$, the category of so-called \textit{perfect complexes}, i.e., each term of $\tilde{P} ^{\bullet}$ is a finitely generated projective $A$-module, and all but finitely many terms of $\tilde{P} ^{\bullet}$ are 0. We can choose $\tilde{P} ^{\bullet}$ such that it is \textit{minimal}. That is, $\tilde{P} ^{\bullet}$ has no direct summands isomorphic to
\begin{equation*}
\xymatrix{ \ldots \ar[r] & 0 \ar[r] & P^s \ar[r]^{id} & P^{s+1} \ar[r] & 0 \ar[r] & \ldots},
\end{equation*}
or equivalently, each differential map in $\tilde{P} ^{\bullet}$ sends its source into the radical of the subsequent term. Clearly, this choice is unique for each $P^{\bullet} \in K^b( _A \mathcal{P})$ up to isomorphism, so in the rest of this note we identify $P^{\bullet}$ and $\tilde{P}^{\bullet}$. Hopefully this identification will not cause trouble to the reader.

Take an indecomposable $P^{\bullet} \in K^b(_A \mathcal{P})$ and identify it with $\tilde{P}^{\bullet}$. Therefore, there exist $r \leqslant s \in \mathbb{Z}$ such that $P^r \neq 0 \neq P^s$, and $P^n = 0$ for $n > s$ or $n < r$. We define its length $l(P^{\bullet})$ to be $s - r$. \footnote{By this definition, the length of a minimal object $X^{\bullet} \in K^b (_AP)$ counts the number of arrows between the first nonzero term and the last nonzero term in $X^{\bullet}$, rather than the number of nonzero terms in $X^{\bullet}$. In particular, a projective $\Lambda$-module, when viewed as a stalk complex in $K^b (_AP)$, has length 0.} The \textit{strong global dimension} of $A$, denoted by $\sgldim A$, is defined as $\sup \{ l(P^{\bullet}) \mid P^{\bullet} \in K^b (_A \mathcal{P}) \text{ is indecomposable} \}$. By taking minimal projective resolutions of simple modules, it is easy to see that $\sgldim A \geqslant \gldim A$. Moreover, if $A$ is hereditary, then $\sgldim A = \gldim A \leqslant 1$ (see \cite{Happel1}). It is not clear for what algebras of finite global dimension, $\sgldim A = \gldim A$.

Happel and Zacharia proved the following result, characterizing piecewise hereditary algebras.

\begin{theorem}
(Theorem 3.2 in \cite{Happel3}) A finite dimensional $k$-algebra is piecewise hereditary if and only if $\sgldim A < \infty$.
\end{theorem}

In \cite{Dionne} Dionne, Lanzilotta, and Smith show that if $\Lambda$ is a piecewise hereditary algebra, and $|G|$ is invertible in $k$, then the skew group algebra $\Lambda G$ is also piecewise hereditary. This result motivates us to characterize general piecewise hereditary skew group algebras. Using the above characterization, we take a different approach by showing that $\sgldim \Lambda G = \sgldim \Lambda$ under suitable conditions.

We first prove a result similar to Lemma 2.4.

\begin{lemma}
Let $H$ be a subgroup of $G$ and suppose that $\gldim \Lambda G < \infty$. Then $\sgldim \Lambda H \leqslant \sgldim \Lambda G$. The equality holds if $|G : H|$ is invertible in $k$.
\end{lemma}

\begin{proof}
Note that $\gldim \Lambda H \leqslant \gldim \Lambda G < \infty$ by (2) of Corollary 2.2 in \cite{Li}. Take an indecomposable object $P^{\bullet} \in K^b (_{\Lambda H} \mathcal{P})$. Up to isomorphism, its minimal preimage in $C^b (_{\Lambda H} \mathcal{P})$ can be written as:
\begin{equation*}
\xymatrix{ \ldots \rightarrow 0 \rightarrow P^r \ar[r]^-{d^r} & \ldots \ar[r] & P^{s-1} \ar[r]^-{d^{s-1}} & P^s \rightarrow 0 \rightarrow \ldots}
\end{equation*}
Applying the exact functor $\Lambda G \otimes _{\Lambda H} -$ termwise to the above complex, we get an object $\tilde{P} ^{\bullet} = P^{\bullet} \uparrow _H^G \in C^b (_{\Lambda G} \mathcal{P})$ as follows:
\begin{equation*}
\xymatrix{ \ldots \rightarrow 0 \rightarrow \Lambda G \otimes P^r \ar[r]^-{1 \otimes d^r} & \ldots \ar[r] & \Lambda G \otimes P^{s-1} \ar[r]^-{1 \otimes d^{s-1}} & \Lambda G \otimes P^s \rightarrow 0 \rightarrow \ldots}
\end{equation*}
Applying the restriction functor termwise, the second complex gives an object $\tilde{P}^{\bullet} \downarrow _H^G \in C^b (_{\Lambda H} \mathcal{P})$. We claim that $P^{\bullet}$ is isomorphic to a direct summand of $\tilde{P} ^{\bullet} \downarrow _H^G$.

Indeed, there is a family of maps $(\iota ^i) _{i \in \mathbb{Z}}$ defined as follows: for $i > s$ or $i < r$, $\iota_i = 0$; for $r \leqslant i \leqslant s$, $\iota^i$ sends $v \in P^i$ to $1 \otimes v \in \Lambda G \otimes _{\Lambda H} P^i$. The reader can check that $(\iota^i) _{i \in \mathbb{Z}}$ defined in this way indeed gives rise to a chain map $\iota^{\bullet}: P^{\bullet} \rightarrow \tilde{P} ^{\bullet} \downarrow _H^G$.

We define another family of maps $(\delta ^i) _{i \in \mathbb{Z}}$ as follows: for $i > s$ or $i < r$, $\delta^i = 0$; for $r \leqslant i \leqslant s$, $\delta^i$ sends $h \otimes v \in \Lambda H \otimes _{\Lambda H} P^i$ to $hv \in P^i$ and sends all vectors in $g \otimes v \in \bigoplus _{1 \neq g \in G/H} g \otimes P^i$ to 0. Again, the reader can check that $(\delta ^i) _{i \in \mathbb{Z}}$ gives rise to a chain map $\delta^{\bullet}: \tilde{P} ^{\bullet} \downarrow _H^G \rightarrow P^{\bullet}$. Moreover, we have $\delta^{\bullet} \circ \iota^{\bullet}$ is the identity map. Therefore, $P^{\bullet}$ is isomorphic to a direct summand of $\tilde{P} ^{\bullet} \downarrow _H^G$.

This claim has the following consequence: for every indecomposable object $X^{\bullet} \in K^b (_{\Lambda H} \mathcal{P})$ (identified with a minimal perfect complex in $C^b (_{\Lambda H} \mathcal{P})$), there is an indecomposable object $\tilde{X}^{\bullet} \in K^b (_{\Lambda G} \mathcal{P})$ such that $X^{\bullet}$ is isomorphic to a direct summand of $\tilde{X} ^{\bullet} \downarrow _H^G$. Clearly, we have $l( \tilde{X} ^{\bullet}) \geqslant l(X^{\bullet})$. Therefore, by definition, $\sgldim \Lambda G \geqslant \sgldim \Lambda H$.

Now suppose that $|G : H|$ is invertible in $k$. Take an indecomposable object $P^{\bullet} \in K^b (_{\Lambda G} \mathcal{P})$. Up to isomorphism, its minimal preimage in $C^b (_{\Lambda G} \mathcal{P})$ can be written as:
\begin{equation*}
\xymatrix{ \ldots \rightarrow 0 \rightarrow P^r \ar[r]^-{d^r} & \ldots \ar[r] & P^{s-1} \ar[r]^-{d^{s-1}} & P^s \rightarrow 0 \rightarrow \ldots}
\end{equation*}
Applying the restriction functor and the induction functor termwise to the above complex, we get another object in $\tilde{P} ^{\bullet} = P^{\bullet} \downarrow _H^G \uparrow _H^G \in C^b (_{\Lambda G} \mathcal{P})$ as follows:
\begin{equation*}
\xymatrix{ \ldots \rightarrow 0 \rightarrow \Lambda G \otimes P^r \ar[r]^-{1 \otimes d^r} & \ldots \ar[r] & \Lambda G \otimes P^{s-1} \ar[r]^-{1 \otimes d^{s-1}} & \Lambda G \otimes P^s \rightarrow 0 \rightarrow \ldots}
\end{equation*}
We claim that $P^{\bullet}$ is isomorphic to a direct summand of $\tilde{P} ^{\bullet}$.

Define a family of maps $(\theta ^i) _{i \in \mathbb{Z}}$ as follows: for $i > s$ or $i < r$, $\theta^i = 0$; for $r \leqslant i \leqslant s$, $\theta^i$ sends $v \in P^i$ to $\frac{1} {|G: H|} \sum _{g \in G/H} g \otimes g^{-1}v \in \Lambda G \otimes _{\Lambda H} P^i$. We check that $(\theta^i) _{i \in \mathbb{Z}}$ defined in this way gives rise to a chain map $\theta^{\bullet}: P^{\bullet} \rightarrow \tilde{P} ^{\bullet}$. Indeed, for $r \leqslant i \leqslant s-1$ and $v \in P^i$, we have
\begin{align*}
(\theta^{i+1} \circ d^i)(v) & = \theta^{i+1} (d^i(v)) = \frac{1} {|G: H|} \sum _{g \in G/H} g \otimes g^{-1} d^i(v) \\
& = \frac{1} {|G: H|} \sum _{g \in G/H} g \otimes d^i(g^{-1}v) = (1 \otimes d^i) (\theta^{i+1} (v)).
\end{align*}

Define another family of maps $(\rho^i) _{i \in \mathbb{Z}}$ as follows: for $i > s$ or $i < r$, $\rho^i = 0$; for $r \leqslant i \leqslant s$, $\rho^i$ sends $g \otimes v \in \Lambda G \otimes _{\Lambda H} P^i$ to $gv \in P^i$. Again, we check that $(\rho^i) _{i \in \mathbb{Z}}$ gives rise to a chain map $\rho^{\bullet}: \tilde{P} ^{\bullet} \rightarrow P^{\bullet}$ as shown by:
\begin{equation*}
(\rho^{i+1} \circ (1 \otimes d^i))(g \otimes v) = \rho^{i+1} (g \otimes d^i(v))) = g d^i (v) = d^i (gv) = d^i (\rho^i (g \otimes v)).
\end{equation*}
Moreover, $\rho^{\bullet} \circ \theta^{\bullet}$ is the identity map. Therefore, as claimed, $P^{\bullet}$ is isomorphic to a summand of $\tilde{P} ^{\bullet}$.

This claim tells us that for every indecomposable object $\tilde{X} ^{\bullet} \in K^b (_{\Lambda G} \mathcal{P})$ (identified with a minimal perfect complex in $C^b (_{\Lambda G} \mathcal{P})$), there is an indecomposable object $X^{\bullet} \in K^b (_{\Lambda H} \mathcal{P})$ such that $\tilde{X} ^{\bullet}$ is isomorphic to a direct summand of $X^{\bullet} \uparrow _H^G$. Therefore, $l(X^{\bullet}) \geqslant l( \tilde{X} ^{\bullet})$, and $\sgldim \Lambda H \geqslant \sgldim \Lambda G$. The two inequalities force $\sgldim \Lambda G = \sgldim \Lambda H$.
\end{proof}

\begin{remark}\normalfont
In this lemma we actually proved a stronger conclusion. That is, the induction and restriction functor induce an `induction' functor and a `restriction' functor between the homotopy categories of perfect complexes. Moreover, every indecomposable object $X^{\bullet} \in K^b (_{\Lambda H} \mathcal{P})$ can be obtained by applying the `restriction' functor to an indecomposable object in $K^b (_{\Lambda G} \mathcal{P})$ and taking a direct summand. When $|G:H|$ is invertible, every indecomposable object $\tilde{X} ^{\bullet} \in K^b (_{\Lambda G} \mathcal{P})$ can be obtained by applying the `induction' functor to an indecomposable object in $K^b (_{\Lambda H} \mathcal{P})$ and taking a direct summand.
\end{remark}

Now we are ready to prove the second part of our main theorem.

\begin{proposition}
Let $\Lambda$, $G$, $S$, and $E$ as before. Then $\Lambda G$ has finite strong global dimension if and only if $\Lambda ^S$ has finite strong global dimension and $S$ acts freely on $E$. In this situation, $\Lambda G$ and $\Lambda ^S$ have the same strong global dimension.
\end{proposition}

\begin{proof}
Since the strong global dimension equals that of $\Lambda S$, without loss of generality we assume that $G = S$. Note that the strong global dimension is always greater than or equal to the global dimension. Therefore, $\Lambda S$ has finite global dimension, so $S$ must act freely on $E$ by (3) of Proposition 2.3. Then by (4) of this proposition, $\Lambda S$ is Morita equivalent to $\Lambda ^S$, so they have the same finite strong global dimension. Conversely, if the action of $S$ on $E$ is free, then again $\Lambda S$ and $\Lambda^S$ are Morita equivalent, so they have the same strong global dimension.
\end{proof}

An immediate corollary of this proposition and Theorem 3.1 is:

\begin{corollary}
Let $\Lambda$, $G$, $S$, and $E$ as before. Then $\Lambda G$ is piecewise hereditary if and only if the action of $S$ on $E$ is free and $\Lambda^S$ is piecewise hereditary.
\end{corollary}

\begin{proof}
By Theorem 3.1, a finite dimensional algebra is piecewise hereditary if and only if its strong global dimension if finite. The conclusion follows from the above proposition.
\end{proof}

In the rest of this section we try to get a get a criterion such that the strong global dimension of $\Lambda$ equals that of $\Lambda S$ when the action of $S$ on $E$ is free. Since we already know that $\Lambda S$ is Morita equivalent to $\Lambda^S$ (Proposition 2.3), instead we show that $\sgldim \Lambda = \sgldim \Lambda ^S$ under a certain condition. Note that $\Lambda^S$ is a subalgebra of $\Lambda$, so we can define the corresponding restriction functor from $\Lambda$-mod to $\Lambda^S$-mod and the induction functor $\Lambda \otimes _{\Lambda^S} -$ from $\Lambda^S$-mod to $\Lambda$-mod.

By \cite{Li}, $\Lambda$ is both a left free and a right free $\Lambda^S$-module, but might not be a free bimodule; see Example 3.6 in that paper. Now assume that $\Lambda = \Lambda^S \oplus B$ as $\Lambda^S$-bimodule. For instance, if $\Lambda^S$ is commutative, this condition is satisfied. Under this assumption, we have a split bimodule homomorphism $\zeta: \Lambda \to \Lambda^S$.

For $M \in \Lambda^S$-mod, we define two linear maps:
\begin{align*}
\psi: M \rightarrow (\Lambda \otimes _{\Lambda ^S} M) \downarrow _{\Lambda ^S} ^{\Lambda}, \quad v \mapsto 1 \otimes v, \quad v \in M;\\
\varphi: (\Lambda \otimes _{\Lambda ^S} M) \downarrow _{\Lambda ^S} ^{\Lambda} \rightarrow M, \quad \lambda \otimes v \mapsto \zeta(\lambda) v, \quad v \in M, \lambda \in \Lambda.
\end{align*}
These two maps are well defined $\Lambda^S$-module homomorphisms (to check it, we need the assumption that $A^S$ is a summand of $\Lambda$ as $\Lambda^S$-bimodules). We also observe that when $\lambda \in \Lambda ^S$,
\begin{equation*}
\varphi (\lambda \otimes v) = \zeta (\lambda) v = \lambda v.
\end{equation*}
Therefore, $\varphi \circ \psi$ is the identity map.

\begin{proposition}
Suppose that a Sylow $p$-subgroup $S \leqslant G$ acts freely on $E$, and $\Lambda \cong \Lambda^S \oplus B$ as $\Lambda^S$-bimodules. Then $\sgldim \Lambda = \sgldim \Lambda S$.
\end{proposition}

\begin{proof}
By Lemma 3.2 and Proposition 3.4, $\sgldim \Lambda \leqslant \sgldim \Lambda S = \sgldim \Lambda^S$. Therefore, it suffices to show $\sgldim \Lambda \geqslant \sgldim \Lambda^S$. The proof is similar to that of Lemma 3.2, using the corresponding induction and restriction functors for the pair $(\Lambda^S, \Lambda)$ and the two maps $\psi, \varphi$ defined above.

Take an indecomposable object $P^{\bullet} \in K^b (_{\Lambda^S} \mathcal{P})$ and write its minimal preimage in $C^b (_{\Lambda^S} \mathcal{P})$ as follows:
\begin{equation*}
\xymatrix{ \ldots \rightarrow 0 \rightarrow P^r \ar[r]^-{d^r} & \ldots \ar[r] & P^{s-1} \ar[r]^-{d^{s-1}} & P^s \rightarrow 0 \rightarrow \ldots}
\end{equation*}
Applying $\Lambda \otimes _{\Lambda^S} -$ and the restriction functor termwise, we get another object in $\tilde{P} ^{\bullet} = P^{\bullet} \uparrow _{\Lambda^S} ^{\Lambda} \downarrow _{\Lambda^S} ^{\Lambda} \in C^b (_{\Lambda^S} \mathcal{P})$ :
\begin{equation*}
\xymatrix{ \ldots \rightarrow 0 \rightarrow \Lambda \otimes P^r \ar[r]^-{1 \otimes d^r} & \ldots \ar[r] & \Lambda \otimes P^{s-1} \ar[r]^-{1 \otimes d^{s-1}} & \Lambda \otimes P^s \rightarrow 0 \rightarrow \ldots}
\end{equation*}
We claim that $P^{\bullet}$ is isomorphic to a direct summand of $\tilde{P} ^{\bullet}$.

As in the proof of the previous lemma, $(\psi_i) _{i \in \mathbb{Z}}$ gives rise to a chain map $\psi^{\bullet}: P^{\bullet} \rightarrow \tilde{P} ^{\bullet}$. We check the commutativity: for $r \leqslant i \leqslant s-1$ and $v \in P^i$,
\begin{equation*}
(\psi_{i+1} \circ d^i)(v) = \psi_{i+1} (d^i(v)) = 1 \otimes d^i(v) = (1 \otimes d^i) (1 \otimes v) = (1 \otimes d^i) (\psi_i (v)).
\end{equation*}
Similarly, the family of maps $(\varphi_i) _{i \in \mathbb{Z}}$ gives rise to a chain map $\varphi^{\bullet}: \tilde{P} ^{\bullet} \rightarrow P^{\bullet}$ as shown by:
\begin{align*}
(\varphi_{i+1} \circ (1 \otimes d^i))(\lambda \otimes v) & = \varphi_{i+1} (\lambda \otimes d^i(v)) = \zeta(\lambda) d^i(v)\\
& =  d^i(\zeta (\lambda) v) = (d^i \circ \varphi_i)(\lambda \otimes v).
\end{align*}
Moreover, $\varphi ^{\bullet} \circ \psi ^{\bullet}$ is the identity map, so as claimed, $P^{\bullet}$ is isomorphic to a summand of $\tilde{P} ^{\bullet}$.

Therefore, for every indecomposable object $X^{\bullet} \in K^b (_{\Lambda^S} \mathcal{P})$, there is an indecomposable object $\tilde{X} ^{\bullet} \in K^b (_{\Lambda} \mathcal{P})$ such that $X^{\bullet}$ is isomorphic to a direct summand of $\tilde{X} ^{\bullet} \downarrow _{\Lambda^S} ^{\Lambda}$. Consequently, $l(X^{\bullet}) \leqslant l( \tilde{X} ^{\bullet})$, and $\sgldim \Lambda^S \leqslant \sgldim \Lambda$.
\end{proof}

\end{document}